\documentclass[11pt,reqno]{amsproc} 
\usepackage{amstext,amsmath,amssymb,amsthm}

\usepackage{tikz}
\usepackage{tikz-3dplot}

\usepackage{latexsym}
\usepackage{exscale}
\usepackage{color}

\newcommand{\C}{  \mathbb  C }
\newcommand{\Z}{  \mathbb Z }
\newcommand{\N}{  \mathbb N }
\newcommand{\Q}{  \mathbb Q }

\newcommand{\supp}{{\rm supp} \, }
\newtheorem{Thm}{Theorem}[section]

\newtheorem{Cor}[Thm]{Corollary}
\newtheorem{Prop}[Thm]{Proposition}
\newcommand{\dsize}{\displaystyle}
\newcommand{\cal}{\mathcal}
\numberwithin{equation}{section}
\newcommand{\E} {{\mathcal E}}
 
\renewcommand{\H} {\mathcal H } 
\renewcommand{\P} {\mathcal P }

\newcommand{\J} {\breve{J}}
\newcommand{\X} {\breve{X}} 
\newcommand{\q} {\cal Q} 
\newcommand{\w}{{\rm w}}

\newcommand{\Mod}[1]{\, (\mathrm{mod}\, #1)}

\theoremstyle{remark}

\newtheorem{Rem}{Remark}[section]

\newcommand{\BEL}{\begin{equation}\label}
	\newcommand{\EE}{\end{equation}}

\begin{document}

	\title[Novel properties of polynomials] {\large  Novel properties of polynomials}

	\author{L. De Carli}
	\address{Laura De Carli, Department of Mathematics, Florida International University,   Miami, FL 33199, USA.}
	\email{decarlil@fiu.edu}

	\author{M.   Laporta}
	\address{Maurizio Laporta, Dipartimento di Matematica e Applicazioni, Università degli Studi di Napoli  "Federico II", 80126 Napoli, Italy}
	\email{mlaporta@unina.it }
	
	

\maketitle

\begin{abstract} 
	We establish necessary and sufficient conditions for a complex polynomial to have a prescribed multiset of roots, expressed in terms of its coefficients and complete homogeneous symmetric polynomials. These conditions extend the classical Vieta formulas and yield explicit formulas for the coefficients of the corresponding quotient. We apply our results to   obtain characterizations of Kronecker polynomials through identities involving symmetric polynomials and construct Borwein multiples of products of cyclotomic polynomials.
 \end{abstract}

  \section{Introduction and statement of the results}\label{hsp}
  
  The main purpose of this paper is to establish new divisibility
  criteria for complex polynomials and explicit relations between
  their coefficients and subsets of roots.  
    We begin by recalling the symmetric polynomials that are involved in 
  this study.
   
  For $\vec z=(z_1,\ldots,z_n)\in\C^n$ and $k\in\N$, the \emph{elementary symmetric polynomial} $\E_k(\vec z\, )$ and the 
  \emph{complete homogeneous symmetric polynomial} $\H_k(\vec z\, )$ are defined as
$$
  \cal E_k(\vec z\, )
  :=
  \sum_{1\le j_1<\cdots<j_k\le n}
  z_{j_1}\cdots z_{j_k},
  \qquad
  \cal H_k(\vec z\, )
  :=
  \sum_{1\le j_1\le\cdots\le j_k\le n}
  z_{j_1}\cdots z_{j_k}.
 $$
  We set $\cal E_0(\vec z\, )=\cal H_0(\vec z\, )=1$ and $\cal E_k(\vec z\, )=0$ when $k>n$.
  Both polynomials are homogeneous of degree $k$  and are invariant by permutation of the variables.
  We have used  \cite{Mcdonald,St}  for the definitions and basic properties of these polynomials, some of which are recalled in   \S 2.  
  
  \smallskip
  The classical Vieta formulas express the coefficients of a complex polynomial
  in terms of the elementary symmetric polynomials evaluated at its roots.
  Specifically, if
  $$
  	\cal P(z)
  	:=a_0\prod_{j=0}^{m-1}(z-x_j)
  	=\sum_{k=0}^{m}a_kz^{m-k},
  	 $$
 with $a_j\in\C$ and $a_0\ne0 $, then
  \begin{equation}\label{e-classic-Viete}
  	\frac{a_k}{a_0}
  	=(-1)^k\cal E_k(x_0,\ldots,x_{m-1}),
  	\qquad 0\le k\le m.
  \end{equation}
 Since   roots of polynomials should be 
 considered with their multiplicities,   the language of  multisets provides the right framework for some of our results and proofs.

 We associate to a given vector
 $\vec y:=(y_0,y_1,\ldots,y_{n-1})\in\mathbb C^n$
 the \emph{multiset} of its components, denoted by
 $$
 Y:=\{y_0,y_1,\ldots,y_{n-1}\}.
 $$
 While the order of the components matters in the vector $\vec y$,
only the values of the elements
 and their multiplicities are important in the multiset $Y$.
 
 Given $y\in\C$, the \emph{multiplicity} ${\rm m}(y;Y)$ is
 the number of occurrences of $y$ in $Y$. Plainly,
 ${\rm m}(y;Y)=0$ if and only if $y$ does not occur in $Y$.
 
 The \emph{support} of $Y$ is the set of its distinct elements.
 It is denoted as ${\rm supp}\,Y$. The \emph{cardinality} of $Y$  is
$
 |Y|:=\sum_{y\in{\rm supp}\,Y}{\rm m}(y;Y)=n.
$

Given multisets $X,\ Y$, their \emph{multiset union},
denoted by $X\sqcup Y$, consists of all $y$ with multiplicity
$
{\rm m}(y;X\sqcup Y):={\rm m}(y;X)+{\rm m}(y;Y)\ge 1.
$

We say that $X$ is a \emph{submultiset} of $Y$, and write
$X\sqsubseteq Y$, if  $\supp X\subseteq\supp Y$ and 
$
{\rm m}(y;X)\le {\rm m}(y;Y)$
 for every $y\in\supp X$.

If $X\sqsubseteq Y$, then the \emph{multiset complement} of $X$
  in $Y$, denoted by $\X$, is formed with  all $y\in\supp Y$ with multiplicity
  $ 
  {\rm m}(y;\X):={\rm m}(y;Y)-{\rm m}(y;X)$.
 Note  that the supports of $X$ and $\X$ may overlap.
  For instance,
  given the multisets $Y:=\{a,\, a,\, b,\, b, c\}$ and $X := \{a, a, b\}$, one has $X \sqsubseteq Y $, 
  $\supp Y =\{a, b, c\}$,  and $\X =\{b,c\}$. 
  
  \smallskip
  Since symmetric polynomials are invariant under permutation
  of their variables, we  will indifferently use  the notation  $\cal E_k(Y)$ and $\E_k(\vec y\, )$ (resp. $\cal H_k(Y)$ and $\H_k(\vec y\, )$).
    We also adopt the convention
  \begin{equation}\label{convention}
  	\cal H_k(\emptyset)=\cal E_k(\emptyset)
  	=
  	\begin{cases}
  		1,&k=0,\\
  		0,&k\ge1.
  	\end{cases}
  \end{equation}

   \subsection{Main results} 
 As already mentioned, when counted with their multiplicity, the roots of a complex polynomial form a multiset, that we call 
  the \emph{roots multiset}.  In this subsection 
  we state explicit relations between
  the coefficients  of a complex polynomial when a submultiset  of its roots multiset  is given. Further, we characterize
  such submultisets.

   \begin{Thm}\label{T-Viete-3} 
   		Let
   		$
   		\P(z)=\sum_{j=0}^{m}a_jz^{m-j}$
   		be a complex polynomial of degree  $m\in\N$ and let $X_m$ be its roots multiset.
   		Given  any multiset  $J_n$  of complex numbers such that
   		$1\le n=|J_n|\le m$,  let us consider the monic polynomial 
   		$$
   		\q_n(z):=\prod_{w\in J_n}(z-w)
   		$$whose roots multiset is $J_n$;
   		 define
   		$$
   		c_k:=\sum_{j=0}^{k}a_j\H_{k-j}(J_n),
   		\qquad 0\le k\le m.
   		$$
   		The following properties are equivalent:
   		\begin{enumerate}
   			\item[(a)] $J_n\sqsubseteq X_m$.
   			\item[(b)] $\q_n(z)$ divides $\P(z)$, and
   			$ \dsize
   			\P(z)=\q_n(z)\sum_{k=0}^{m-n}c_kz^{m-n-k}.
   			$ 
   			\item[(c)] $c_k=0$ for every $k\in\{m-n+1,m-n+2,\ldots,m\}$.
   		\end{enumerate}
   \end{Thm}
   This theorem  considerably improves on
 \cite[Th.\ 1.1]{DEL}, where we have just established that the property (b) implies (c), whereas here we show that the converse holds.
 
 Further, the
  Vieta formulas \eqref{e-classic-Viete} are recovered by 
  taking $J_n=X_m$, i.e., $n=m$, in
  Theorem \ref{T-Viete-3}. In order to see it, let us first recall a well known identity relating complex and homogeneous  symmetric polynomials given by
  \begin{equation}\label{e-wellknownidentity}
  	\sum_{j=0}^{k}(-1)^j
  	\cal E_j(\vec z\, )\cal H_{k-j}(\vec z\, )=0,
  	\qquad (k\in \N).
  \end{equation}
  See Proposition \ref{P-Main} below for a more general identity.
  Then, note that for $J_n=X_m$ the condition (c) of Theorem \ref{T-Viete-3} becomes
  \begin{equation}\label{(3)X_m}
   \sum_{j=0}^k a_j\H_{k-j}(X_m)=0,
  \qquad k=1,\ldots,m.
  \end{equation}
In particular, for $k=1$ this gives
$$
a_0\H_1(X_m)+a_1\H_0(X_m)=a_0\H_1(X_m)+a_1=0,
$$ 
that is 
 $\H_1(X_m)=\E_1(X_m)=-a_1/a_0$, i.e., the formula \eqref{e-classic-Viete} for $k=1$. 

Now, for $k=2$ observe that \eqref{e-wellknownidentity} yields
$\H_2(X_m)=\E_1(X_m)^2-\E_2(X_m)$. Thus, from \eqref{(3)X_m} for $k=1,2$ it follows that
$$
\frac{a_2}{ a_0}=-\H_2(X_m)-a_1\H_1(X_m)/a_0=
-\H_2(X_m)+\E_1(X_m)^2=\E_2(X_m),
$$  
which is \eqref{e-classic-Viete} for $k=2$. The remaining case $3\le k\le m$ can be accomplished in a completely analogous way.

  \smallskip
     
  Before presenting some consequences of Theorem~\ref{T-Viete-3},
  let us introduce the following notation.  {Given $ h\in\{1,2,\ldots,m\}$}, 
  we group  the terms of the complex polynomial 
  $ \dsize 
  \P(z):=\sum_{k=0}^{m}a_kz^{m-k}$, with $ {\rm deg}\P(z)=m\in\N,
 $ 
 according to their exponents modulo 
  $h$, so that
  $$
  \P(z)=\sum_{s=0}^{h-1}\sum_{j=0\atop j\equiv s\Mod h}^ma_{m-j}z^j:=\sum_{s=0}^{h-1}z^s\P_{h,s}(z^h),
$$
  where
  \begin{equation}\label{e-def-Ps}
  	\P_{h,s}(z)
  	:=
  	\sum_{q=0}^{[(m-s)/h]}
  	a_{m-qh-s}z^q,
  	\qquad 0\le s<h.
  \end{equation}
  Hereafter, $[x]$ denotes the integer part of the real number $x$.

 Note that 
  $$
  0\le \left[\frac{m+1}{h}\right]-1\le {\rm deg} \P_{h,s}(z)=\left[\frac{m-s}{h}\right]\le \left[\frac{m}{h}\right]\le m.
  $$

   \begin{Cor}\label{C-h-sections} Let  $\P(z)$ defined as before and let $h\in\{1,2,\ldots,m\}$ be fixed.  
   	We have that
   	$\cal G(z^h)$ divides $\cal P(z)$ 
   	if and only if
   	$ 
   	\cal G(z)$ divides $\cal P_{h,s}(z)$
   	for every $s\in\{0,1,\ldots,h-1\}$. 
   	  \end{Cor}

   \smallskip

    The next corollary characterizes the divisibility of
    $\P(z)$ by the complex polynomial 
    $ \dsize
 \q(z)= \sum_{j=0}^{n}b_jz^{n-j},\ \hbox{with}\  n\leq m,
 $ 
 in terms of their coefficients.

   Let us first recall that, given $k\in\N_0$ and  $\vec\alpha:=(\alpha_1,\ldots,\alpha_n)\in\mathbb N_0^n$, the \emph{multinomial coefficient} is defined as \cite[\S 1.2]{St1}
   	$$
   	  \binom{k}{\alpha_1,\ldots,\alpha_n}:= 
   	  \frac{k!}{\alpha_1!\cdots\alpha_n!}=
   \binom{k}{\alpha_1}
   \binom{k-\alpha_1}{\alpha_2}
   \cdots
   \binom{k-\alpha_1-\cdots-\alpha_{n-1}}{\alpha_n}
  .
   $$
   Further, we
    introduce the following notation:
   $$
   \Sigma(\vec\alpha):=\sum_{j=1}^{n}\alpha_j,
   \qquad
   \w(\vec\alpha):=\sum_{j=1}^{n}j\alpha_j.
   $$

   \begin{Cor}\label{C-divis-cond}
    Let $\P(z)= \sum_{k=0}^{m}a_kz^{m-k}$  and $\q(z)=\sum_{k=0}^{n}b_kz^{n-k}$ be  complex polynomial, with  $1\leq n\leq m$.  We have that $\q(z)$ divides $\P(z)$ if and only if
   \begin{align*}
    a_k+\sum_{\ell=1}^{k}(-1)^\ell
    \sum_{\substack{1\le m_1,\ldots,m_\ell\le n }}
    a_{k-(m_1+\cdots+m_\ell)}
    b_{m_1}\cdots b_{m_\ell}&=\\
   	\sum_{\substack{\vec\alpha\in\mathbb N_0^n\\\w(\vec\alpha)\le k}}
    	(-1)^{\Sigma(\vec\alpha)}
   	\binom{\Sigma(\vec\alpha)}{\alpha_1,\ldots,\alpha_n}
   	a_{k-\w(\vec\alpha)}
   	\prod_{j=1}^{n}b_j^{\alpha_j}&=0
  \end{align*}
  for every $k\in\{m-n+1,\ldots,m\}$.
   \end{Cor}

 \subsection{Further results}\label{furtherresults}  In \cite{DL3} we have recently proved that if 
 the roots multiset $X_m$ of a complex and monic
 polynomial  $\P(z)$ consists of $m={\rm deg}\cal P(z)$ distinct $n$th roots of unity, then  the cofactor of $\P(z)$
 in $z^n-1$ is given by 
$$
\frac{z^n-1}{\P(z)}=\sum_{j=0}^{n-m}\H_j(X_m)z^{n-m-j}.
$$

 This result motivates the more general study of the product $\P(z)\q_r(z)$ when 
  $X_m$ is  an arbitrary root multiset and
$$
 \q_r(z):=\sum_{s=0}^r\H_s(X_m)z^{r-s},
 \ \hbox{with}\ {\rm deg}\q_r(z)=r\in\N_0.
$$
The next proposition identifies $\q_r(z)$ as the
 quotient in the division of $z^{m+r}$ by $\P(z)$ and
 gives an explicit formula for the remainder.
 An equivalent formula appears in
 \cite[Eq. (18)]{WuYang}; we include a short proof of our proposition in \S \ref{Ths}.

  \begin{Prop}\label{P-product2poly}
 	Let $\P(z)=\sum_{j=0}^m a_jz^{m-j}$ be a complex and monic
 	polynomial of degree $m\in\N$ with roots multiset $X_m$.
 	For any given $r\in\N_0$ we have that 	$\P(z)\q_r(z)=z^{m+r} +\cal R_r(z)$, where
 $$
 	\q_r(z):=\sum_{s=0}^r\H_s(X_m)z^{r-s},\qquad 
 	\cal R_r(z):=\sum_{k=1}^{m }c_{k,r}z^{m-k} ,
 $$
\begin{equation}\label{e-ckr}
 \mbox{with} \hskip 1.5 cm 	c_{k,r}:=-\sum_{j=0}^{k-1}a_j\H_{k+r-j}(X_m),
 	\quad k\in\{1,\ldots,m\}.
 \end{equation}
 In particular, for any given $r\in\N_0$ and $s\in\C$, the polynomial $\P(z)$ divides
 $z^{m+r}-s$ if and only if
 $$
 \sum_{j=0}^{k-1}a_j\H_{k+r-j}(X_m)=	\begin{cases}
 	0,& \mbox{if  $k\in\{1,\ldots,m-1\}$},\\
 	s,& \mbox{if  $k=m$}.
 \end{cases}
 $$
 \end{Prop}

 \begin{Rem}\label{C-binomial-divisibility} By taking $s=0$ we see that $\P(z)$ divides
 	$z^{m+r}$ if and only if $c_{k, r}=0$ for every $k\in\{1\ldots,m\}$. In other words,
 	if the roots multiset $X_m$ of $\P(z)$ is such that ${\rm supp}X_m\not=\{0\}$, then for every
 	$r\in\N_0$ there exists $k\in\{1,\ldots,m\}$ such that
 	$$
 	\sum_{j=0}^{k-1}a_j\H_{k+r-j}(X_m)=\sum_{j=0}^{k-1}(-1)^j\E_j(X_m)\H_{k+r-j}(X_m)\ne0.
 	$$
 	 This should be compared to the standard identity \eqref{e-wellknownidentity}.
 \end{Rem}

  		\subsection{A criterion for Kronecker polynomials}

  		A monic polynomial
  		with integer coefficients is a \emph{Kronecker polynomial} if all
  		its zeros belong to the closed unit disk. A classical theorem of
  		Kronecker shows that every nonzero root of such a polynomial is in
  		fact a root of unity \cite{K}. Consequently, if $\cal K(z)$ is a  Kronecker polynomial, then it
  		admits a factorization of the form
  		\begin{equation}\label{e-Kronecker-factorization}
  			\cal K(z)
  			=
  			z^r\prod_{j\in K}\Phi_j(z)^{b_j},
  		\end{equation}
  		where $K\subset\mathbb N$ is finite, $r\in\mathbb N_0$,
  		$b_j\in\mathbb N$, and $\Phi_j(z)$ is the cyclotomic
  		polynomial of index $j$. Conversely, every polynomial of the form
  		\eqref{e-Kronecker-factorization} is a Kronecker polynomial.
  		
  		 Recall that the cyclotomic polynomial of index $n\in\N$ is defined as
  		$$
  		\Phi_n(z):=\prod_{s=1\atop \gcd(s,n)=1}^n(z- e^{2\pi is/n} ),
  		$$
  		with ${\rm deg}	\Phi_n(z)=\varphi(n):=|\{s\in\N:\, 1\le s\le n\, \hbox{and}\,  \gcd(s,n)=1\}|$, the Euler totient function.
  		It is well known that its coefficients belong to $\Z$, it is irreducible over $\Q$, and 
  		$ \dsize
  		z^n-1=\prod_{d\mid n}\Phi_d(z).
  		$ 
  	We refer the reader to   \cite{DF} \S 13.6  and also to the survey papers \cite{Sanna, T}  for an overview of these polynomials. 
  		
  		\medskip
  		 
  		 The following proposition characterizes Kronecker polynomials whose roots are nonzero. It is used to derive the criterion of Theorem \ref{T-kronekerequiv-alt} below. In \S \ref{Ths} we include a proof of such a proposition.

  	 \begin{Prop}\label{P-Kronecker-divisibility} 
  	 	Let $\cal K(z)$ be a monic polynomial with integer coefficients such that
  	 	$\cal K(0)\ne0$. The following properties are equivalent:
  	 	\begin{enumerate}
  	 		\item[(a)] $\cal K(z) $ is a Kronecker polynomial.
  	 		\item[(b)] 
  	 		$\cal K(z)$ divides $\cal K(z^t)$ for some  integer $t>1$.

  	 		\item[(c)] 
  	 		$\cal K(z)$ divides $(z^N-1)^s$ for some  integers $N>1$ and $s\ge 1$.
  	 	\end{enumerate}
  	 \end{Prop}
\begin{Rem}
  		Ayache, Echi, and Naimi introduced in  \cite {AEN}  the more restrictive class of absolutely Kronecker polynomials characterized by the condition that
  		$\P(z)$ divides $\P(z^t)$
  			in $\C[z]$
  		for every $t\in\N$.\end{Rem}
  	\smallskip

  Here we can	combine Proposition \ref{P-Kronecker-divisibility} with  the  divisibility criteria given by  Corollary \ref{C-divis-cond} to obtain the following result, which we prove in \S \ref{Ths}.

     \begin{Thm}\label{T-kronekerequiv-alt}
     	Let
     	$\dsize 
     	\cal K(z)=\sum_{j=0}^n a_jz^{n-j}
     	$
     	be a monic polynomial of degree $n\in\N$ with
     	integer coefficients and $a_n=\mathcal K(0)\ne0$.
     	 Define
     	$$
     	F_0:=1\quad \hbox{and}\quad 
     	F_k:=\sum_{s=1}^k(-1)^s\hskip - 4mm
     	\sum_{\substack{j_1+\cdots+j_s=k\\j_i\in\N}}
     	a_{j_1}\cdots a_{j_s} 
     	\quad \hbox{when}\ k\in\N.
     	$$
     	The following properties are equivalent:
     	\begin{enumerate}\item[(a)] $\cal K(z)$ is a Kronecker polynomial.
     	\item[(b)]  There exists an integer  $t>1$ such that \begin{equation}\label{e-sum-K}\sum_{j=0}^{[k/t]}
     		a_j\H_{k-tj}(X_n)=0,
     		\qquad (t-1)n<k\le tn,\end{equation} 
     		where  $X_n$ is the roots multiset of $\cal K(z)$.
     		\item[(c)] There exists an integer $t>1$ such that
     		\begin{equation}\label{e-Fk=0}
     			\sum_{j=0}^{[k/t]}
     			a_jF_{k-tj}=0,
     			\ \hbox{when}\ (t-1)n<k\le tn.
     		\end{equation}
      \end{enumerate}
     \end{Thm}

  \subsection{Applications to polynomials with restricted coefficients}
  
 Given a polynomial $\P(z)$ with integer coefficients and a finite
 set $D\subset\mathbb Z$, a question to ask is whether $\P(z)$
 divides a nonzero polynomial whose coefficients
 belong to $D$. Here are some cases of particular interest:
 
 \begin{itemize}
 	\item  
 	$D=\{0,1\}$ and constant term $1$ (\emph{Newman polynomials})
 	\item  $D=\{-1,1\}$ (\emph{Littlewood polynomials})
 	\item 
 	$D=\{-1,0,1\}$ and nonzero constant term (\emph{Borwein polynomials})
 \end{itemize}
 
  A major issue is to determine whether such
  multiples exist and, when they do, to find their
  smallest possible degree. These questions have
  been studied in connection with polynomials of
  small height, Mahler measure, Pisot and Salem
  numbers, and finite-state algorithms; see  
  \cite{Borwein},  and also \cite{Mossinghoff, DJS},  the recent \cite{MI} and the references therein.
  
   \medskip 
 By applying Corollary \ref{C-divis-cond} we easily obtain the next divisibility criterion which yields explicit
  conditions for the existence of multiples of
  $\P(z)$ with
  coefficients in a prescribed finite set of integers. Indeed, 
 after  recalling the
  	definition of $F_j$ given in Theorem \ref{T-kronekerequiv-alt}, it suffices to observe that the divisibility of
  	$\P(z)$ by 	the polynomial $\dsize
  	\sum_{j=0}^N b_jz^{N-j}$,
  	with $b_0:=1$,
  	is equivalent to
  	\eqref{e-NeS-coeff} below. Thus, we get immediately the following criterion.

  \begin{Thm}\label{C-prescribed-coeff}
  	Let
  	$
  	\P(z)=\sum_{j=0}^m a_jz^{m-j}
  	$
  	be a monic polynomial of degree
  	$m\in\N$, with integer coefficients and roots  multiset $X_m$.

  	The polynomial $\P(z)$ has a 
  	multiple of degree $N\ge m$, with coefficients in a finite set $D\subset\mathbb Z$,
  	if and only if there exist
  	$b_1,\ldots,b_N\in D$ such that
  	\begin{equation}\label{e-NeS-coeff}
  		\sum_{j=0}^k b_{k-j}\H_j(X_m)=0,
  			\ \hbox{when}\ N-m<k\le N,
  	\end{equation}
  	where we have set $b_0:=1$.
  \end{Thm}
  In view of Proposition \ref{P-H-E}, all the equations \eqref{e-NeS-coeff} can be written  in  terms of the coefficients of $\P(z)$.

  By taking $D=\{0,1\}$, $\{-1,1\}$, or
  $\{-1,0,1\}$ in the previous theorem, we obtain a
  characterization of the existence of Newman,
  Littlewood, or Borwein multiples,
  respectively, with the additional requirement
  $b_N=1$ in the Newman case and $b_N\ne0$
  in the Borwein one.

   Plainly, such a characterization consists of a finite system of linear equations with solutions in a preassigned set $D\subset\Z$. Consequently, the
   minimum-degree problem can be formulated
   as finding the smallest $N\ge m$ for which
   \eqref{e-NeS-coeff} admits solutions in $D$.
   This formulation might help in
   computational investigations of Newman,
   Littlewood, and Borwein multiples.
  
  \smallskip

  	It is proved in \cite[Proposition~13]{DJS} that every product of cyclotomic polynomials has a Littlewood multiple.  Since a Littlewood polynomial is a Borwein polynomial, the existence of a Borwein multiple of a product of cyclotomic polynomials is not new. However, here we present the following property 
 
   \begin{Prop}\label{generalpropertyBorwein}
   	If 
   	$  \dsize
   	\P(z)=\prod_{\ell=1}^k\Phi_{j_\ell}(z)^{d_\ell},
   	$ 
   	where the indices $j_\ell$ are distinct and $d_\ell\in\N$,
   	then $\P(z)$ divides  a Borwein
   	polynomial of the form 
   	$$
   	\cal B(z):=\prod_{s=1}^n(z^{N_s}-1),
   	$$
   	where 
   	$n:=\max_{1\le \ell\le k}d_\ell$ and
   	$N_1,\ldots,N_n\in\N$ satisfy   {$N_1=\operatorname{lcm}(d_1\cdots d_k)$ if $s=1$ and }
   	 \begin{equation}\label{e-ineq1}
   		N_s>\sum_{i=1}^{s-1}N_i,
   			\ \hbox{when}\ 2\le s\le n.
   	\end{equation}
   \end{Prop}

   	This proposition and Proposition \ref{P-Kronecker-divisibility} 
   	immediately yield the following
   	
  \begin{Cor}\label{T-Kron-Bor}
   {Every Kronecker polynomial $\cal K(z)$, with $\cal K(0)\ne 0$, has a Borwein multiple}
  \end{Cor}

The proofs of the aforementioned results are given in \S \ref{Ths}. For such proofs we need some old and new properties of the symmetric polynomials $\H_m$ and $\E_m$ which are collected  in the next section.

 \section{Old and new identities for symmetric polynomials}

 \subsection{Generating functions}
 
 For $\vec z=(z_1,\ldots,z_n)\in\C^n$, the generating functions
 of  $\H_m$ and  $\E_m$ 
 are
 \begin{equation}\label{genH}
 	H(t;\vec z\, )
 	:=\prod_{j=1}^{n}(1-tz_j)^{-1}
 	=\sum_{m=0}^{\infty}\H_m(\vec z\, )t^m
 \end{equation}
 and
 \begin{equation}\label{genE}
 	E(t;\vec z\, )
 	:=\prod_{j=1}^{n}(1+tz_j)
 	=\sum_{m=0}^{n}\E_m(\vec z\, )t^m,
 \end{equation}
respectively. We regard  \eqref{genH} as an  identity of formal power series.
 
 As before, if  $Y_n$ is the multiset of the components of $\vec y$, then
 we write
 $H(t;Y_n)=H(t;\vec y\, )$ and $E(t;Y_n)=E(t;\vec y\, )$.

In view of \eqref{genH} and \eqref{genE}, one has
 \begin{equation}\label{e-gen-funct1}
 	\H_m(\vec z\, )
 	=\frac{1}{m!}
 	\frac{{\rm d}^m}{{\rm d}t^m}H(t;\vec z\, )_{\vert_{t=0}},
 	\qquad
 	\E_m(\vec z\, )
 	=\frac{1}{m!}
 	\frac{{\rm d}^m}{{\rm d}t^m}E(t;\vec z\, )_{\vert_{t=0}},
 \end{equation}
 for every $m\ge0$.
 Though the first formula  differentiation is understood formally,  the usual product
 and chain rules remain valid \cite{St}.
 Moreover,
 \begin{equation}\label{e-prod-EH}
 	E(-t;\vec z\, )H(t;\vec z\, )
 	=
 	H(-t;\vec z\, )E(t;\vec z\, )=1.
 \end{equation}
 
  The properties of  the generating functions given in Propositions \ref{L-pol-E} and \ref{L-generating-products} below are well known,  but we  provide their short proofs for convenience.
 The first property relates these generating functions to a polynomial
 with prescribed roots multiset.
 
 \begin{Prop}\label{L-pol-E}
 	Let $\P(z)=\sum_{j=0}^{m}a_jz^{m-j}$, with $a_0\ne0$,
 	be a complex polynomial with roots multiset $X_m$. Then
 	$$
 	t^m\P(t^{-1})=a_0E(-t;X_m),
 	$$
 	and consequently
 	$$
 	t^m\P(t^{-1})H(t;X_m)=a_0.
 	$$
 \end{Prop}
 
 \begin{proof}
 	The factorization
 	$\P(z)=a_0\prod_{x\in X_m}(z-x)$ yields
 	$$
 	t^m\P(t^{-1})
 	=a_0\prod_{x\in X_m}(1-tx)
 	=a_0E(-t;X_m).
 	$$
 	The second identity follows from \eqref{e-prod-EH}.
 \end{proof}
 
 The next proposition extends  the identity \eqref{e-prod-EH}.   
 
 \begin{Prop}\label{L-generating-products}
 	Let $\vec x=(x_0,\ldots, x_{k-1})\in\C^k$, $\vec y=(y_0,\ldots, y_{n-1})\in\C^n$, and denote
 	$(\vec x, \vec y\, )=(x_0,\ldots, x_{k-1}, y_0,\ldots, y_{n-1})$.
 	The following identities of formal power series hold:
 	$$
 	H(t;\vec x\, )=E(-t;\vec y\, )H(t;\vec x,\vec y\, ),
 	\qquad
 	E(t;\vec x\, )=H(-t;\vec y\, )E(t;\vec x,\vec y\, ).
 	$$
 \end{Prop}
 
 \begin{proof}
 Since
 	$$
 	H(t;\vec x,\vec y\, )=H(t;\vec x\, )H(t;\vec y\, ),
 	\qquad
 	E(t;\vec x,\vec y\, )=E(t;\vec x\, )E(t;\vec y\, ),
 	$$
 	the claimed identities follow immediately by
 	multiplying the first equation by $E(-t;\vec y\, )$,
 	the second one by $H(-t;\vec y\, )$, and  using
 	\eqref{e-prod-EH}.
 \end{proof}
 
 \subsection{Convolution identities}

The identities in the next proposition can be found in \cite[Chapter I, §5, Example 25(a), p. 92.]{Mcdonald}.   See also   \cite[Th.\ 1.1]{Merca}.

\begin{Prop}\label{1L-new-id-hm}
	Let $m\in\N_0$ and $\vec v_1\in \C^{d_1},\ldots,\vec v_n\in \C ^{d_n}$. We have that
		$$
			\H_m (\vec v_1,\dots,\vec v_n)=\sum_{m_1+\ldots+m_n=m\atop m_i\ge 0}\H_{m_1}(\vec v_{1})\cdots\H_{m_n}(\vec v_n).
		$$
		$$
			\E_m (\vec v_1,\dots,\vec v_n)=\sum_{m_1+\ldots+m_n=m\atop m_i\ge 0 }\E_{m_1}(\vec v_{1})\cdots\E_{m_n}(\vec v_n).
		$$
\end{Prop}

In the proof of Corollary \ref{C-divis-cond} we apply the following identities for $\H_m$ and $\E_m$. 
They follow easily by comparing coefficients of $t^m$ in the formal power
series after 
using \eqref{genH}, \eqref{genE}, and Proposition \ref{L-generating-products}.

\begin{Prop}\label{P-Main}  Let $m \in\N_0$. For all
	$\vec x\in\C^k$, $\vec y\in\C^n$ we have that
	$$
		\H_m(\vec x )=\sum_{j=0}^m (-1)^{j}\E_{j}(\vec y\,)\H_{m-j}(\vec x,\, \vec y\,),
	$$
$$
		\E_m(\vec x  )=\sum_{j=0}^m (-1)^{m-j}\H_{m-j}(\vec y\,)\E_{j}(\vec x,\, \vec y\,).
	$$
\end{Prop}
Both identities in Proposition~\ref{P-Main} are special cases of the classical
formulas for the symmetric functions of a \emph{difference of alphabets}
(finite multisets of variables).
Writing $h_m(X)=\H_m(\vec x\,)$ and $e_m(X)=\E_m(\vec x\,)$ for the alphabet
$X=\{x_1,\dots,x_k\}$, these formulas read
$$
h_m(X-Y)\!=\!\sum_{j=0}^m(-1)^j e_j(Y) h_{m-j}(X),
$$
$$
e_m(X-Y)\!=\!\sum_{j=0}^m(-1)^{m-j} h_{m-j}(Y) e_j(X) 
$$
and Proposition~\ref{P-Main} follows by replacing $X$ with the alphabet
$X\sqcup Y=\{x_1,\dots,x_k,y_1,\dots,y_n\}$. For more details, see e.g. \ \cite{Las, Z}.

\medskip
Note  also that Proposition \ref{P-Main}  can  be equivalently   reformulated in terms of multisets as follows. 
Let $J\sqsubseteq Y$, with $1\le |J|\le |Y|$, and let
$\J$ be the complement multiset of $J$ with respect to $Y$. 
For every $m \in\N_0$ one has
$$
	\H_m(J)=  \sum_{j=0}^{m}(-1)^j\E_j( \J )\H_{m-j}(Y),
$$
$$
	\E_m(J)=  \sum_{j=0}^{m} (-1)^ {m-j}\H_{m-j}( \J)\E_j(Y).
$$
Applying the latter formula with $J$ and its complement
$\J$ interchanged gives
$$
\E_m(\J)
=
\sum_{j=0}^{m}(-1)^j\H_j(J)\E_{m-j}(Y),
$$
yielding in particular 
$\dsize
\sum_{j=0}^{m}(-1)^j\H_j(J)\E_{m-j}(Y)=0$
if $m>|Y|-|J|$. 

This reduces to the standard identity
\eqref{e-wellknownidentity} when $J=Y$.

\subsection{Expressing $\H_m$ in terms of $\E_k$}

In the proof of Corollary \ref{C-divis-cond}
we apply the following property, which is a special case of a   theorem  by  Gauss (see \cite[\S 2.2.15]{B}).

\begin{Prop}\label{P-H-E}
	For every $m\in\N $ and $\vec z\in\C^n$ one has
	\begin{align}\label{e-f1}
		\H_m(\vec z\, ) &= 
		\sum_{k=1}^m (-1)^{m-k}\hskip-3mm\sum_{\substack{m_1+\cdots+m_k=m\\ 1\le m_i\le n}}
		\E_{m_1}(\vec z\,)\cdots\E_{m_k}(\vec z\,) 
		\\ \label{e-f2}
		&=\sum_{\substack{\vec\alpha\in\mathbb N_0^n\\\w(\vec\alpha)=m}}
		(-1)^{m-\Sigma(\vec\alpha)}
		\binom{\Sigma(\vec\alpha)}{\alpha_1,\ldots,\alpha_n}
		\prod_{j=1}^{n}\E_j(\vec z\, )^{\alpha_j},
	\end{align}
	where
	$$
	\Sigma(\vec\alpha):=\sum_{j=1}^{n}\alpha_j,
	\qquad
	\w(\vec\alpha):=\sum_{j=1}^{n}j\alpha_j.
	$$
\end{Prop}

This  proposition  shows that if the $\E_j(\vec z\, )\in\Z$ for all $j=0,\ldots, n$, then also  $\H_m(\vec z\, )\in\Z$  for all $m\in\N_0$.

For the proof of Proposition \ref{P-H-E} we need the multinomial theorem, which states that \cite[\S 1.2]{St1}
$$
(x_1+\ldots +x_n)^k=
\sum_{\vec\alpha:=(\alpha_1,\ldots,\alpha_n)\in\N_0^n\atop	\Sigma(\vec\alpha)=k}
\binom{k}{\alpha_1,\ldots,\alpha_n}
x_1^{\alpha_1}\cdots x_n^{\alpha_n}
$$ 
for any $x_1,\ldots,x_n\in\C$, $k\in\N_0$. Let us recall also the generalized Leibnitz formula 
$$
	\frac{{\rm d}^m}{{\rm d}t^m}\prod_{j=1}^kf_j(t)
	=
	\sum_{\substack{m_1+\dots+m_k=m\\ m_i\in\N_0}}
	\binom{m}{m_1,\ldots,m_n}
	\prod_{j=1}^k	\frac{{\rm d}^{m_j}}{{\rm d}t^{m_j}}f_j(t).
$$ 
\begin{proof}[Proof of Proposition \ref{P-H-E}] Clearly, both identities \eqref{e-f1} and \eqref{e-f2} are trivially true when $m=0$. Thus, we assume that $m\in\N$.
	
	By using \eqref{genH} and \eqref{genE}, we write 
	\begin{align*}
		H(t; \vec z\,)&=E(-t; \vec z\,)^{-1}=\left(\sum_{j=0}^n(-1)^j\E_j(\vec z\, )\, t^j\right)^{-1}\\
		&=
		\frac{1}{1- (\E_1 t-\E_2 t^2 +\cdots+ t^n(-1)^{n-1}\E_n)} :=
		\frac{1}{1- F_n(t;\vec z\,)},
	\end{align*}
	where we have let 
	$$
	\E_j:=\E_j(\vec z\,)\quad \hbox{and}\quad 
		F_n(t;\vec z\,):=   \E_1 t-\E_2 t^2 +\cdots+ (-1)^{n-1}t^n\E_n,
		$$
	for brevity.  Thus, expanding   as a formal power series gives  
	$$
	H(t; \vec z\,) =\sum_{k=0}^\infty F_n(t;\vec z\,)^k,
	$$ 
	so that \eqref{e-gen-funct1} yields
	$ \dsize
	\H_m(\vec z\,)=\frac{1}{m!}\sum_{k=0}^\infty \frac{{\rm d}^m}{{\rm d}t^m} F_n(t;\vec z\,)^k_{\vert_{t=0}}.
	$

	Now, observe 
	that $F_n^{(0)}(0;\vec z\,):=F_n(0;\vec z\,)=0$ and
	$$
	F_n^{(m)}(0;\vec z\,):=\frac{{\rm d}^m}{{\rm d}t^m}F_n(t;\vec z\,)\vert_{t=0}=
		\begin{cases}
		m! (-1)^{m+1}\E_m(\vec z\,), &\mbox{if $1\le m\le n$},\\
		0,& \mbox{otherwise}.
	\end{cases}
	$$
	Therefore, by applying the generalized Leibnitz formula we see that
	\begin{align*}
		\frac{{\rm d}^m}{{\rm d}t^m}F_n(t;\vec z\,)^k\vert_{t=0}
		&=m!
		\sum_{\substack{m_1+\dots+m_k=m\\ 1\le m_i\le n}}
		\frac{F_n^{(m_1)}(0;\vec z\,)\cdots F_n^{(m_k)}(0;\vec z\,)}{m_1!\cdots m_k!}
		\\
		&=
		(-1)^{m+k}m!\sum_{\substack{m_1+\dots+m_k=m\\ 1\le m_i\le n}}
		\E_{m_1}(\vec z\,)\cdots\E_{m_k}(\vec z\,).
	\end{align*}
	Since $m_1+\dots+m_k=m$ can hold with $m_i\in\N$ only  when $1\le k\le m$, the previous series reduces to a finite sum as
	\begin{align}\label{finitereduction}
		\H_m(\vec z\,)&=\frac{1}{m!}\sum_{k=1}^m \frac{{\rm d}^m}{{\rm d}t^m} F_n(t;\vec z\,)^k_{\vert_{t=0}}\\
		&= \sum_{k=1}^m(-1)^{m+k}\hskip -3mm\sum_{\substack{m_1+\dots+m_k=m\\ 1\le m_i\le n}}
		\E_{m_1}(\vec z\,)\cdots\E_{m_k}(\vec z\,),\nonumber
	\end{align}
	which is \eqref{e-f1}.

	In order to prove \eqref{e-f2}, first note that the multinomial theorem yields
		\begin{align*}
	F_n(t;\vec z\,)^k:=&\ (-1)^k(\E_1(-t)+\E_2(-t)^2 +\cdots+ (-t)^{n}\E_n)^k\\
	=&\	(-1)^k\sum_{\substack{\vec\alpha\in\mathbb N_0^n\\ \Sigma(\vec\alpha)=k}}
	\binom{k}{\alpha_1,\ldots,\alpha_n}
	\prod_{j=1}^n \E_j(\vec z\, )^{\alpha_j}\,(-t)^{\w(\vec\alpha)}.
		\end{align*}
	Since for any $\vec\alpha\in\mathbb N_0^n$  one has
	$$
	\frac{{\rm d}^m}{{\rm d}t^m}(-t)^{\w(\vec\alpha)}=
	\begin{cases}
		(-1)^{\w(\vec\alpha)}\frac{\w(\vec\alpha)!}{(\w(\vec\alpha)-m)!}t^{\w(\vec\alpha)-m}, &\mbox{if $1\le m\le \w(\vec\alpha)$},\\
		0,& \mbox{otherwise,}
	\end{cases}
	$$
	from \eqref{finitereduction} we see that
$$
		\H_m(\vec z\,)=
	(-1)^m\sum_{k=1}^m(-1)^k\sum_{\substack{\vec\alpha\in\mathbb N_0^n\\
			\Sigma(\vec\alpha)=k,\, \w(\vec\alpha)=m}}
	\binom{k}{\alpha_1,\ldots,\alpha_n}
	\prod_{j=1}^n\E_j(\vec z\, )^{\alpha_j}.
	$$
	Hence, this reduces to \eqref{e-f2} because $\Sigma(\vec\alpha)\le \w(\vec\alpha)$ for any $\vec\alpha\in\mathbb N_0^n$, so that there exists  exactly one value
	$k=\Sigma(\vec\alpha)\in\{1,\ldots,m\}$ such that $\w(\vec\alpha)=m$. 
	
	The proposition is proved.
\end{proof}

 \begin{Rem}\label{C-Jacobi-TrudyCor} 
From \eqref{finitereduction} we obtain the following coefficient
 formula: if $X_m$ is the roots multiset of the complex polynomial $\P(z)=\sum_{j=0}^m a_jz^{m-j}$, then
 	$$
 	\H_n(X_m)=\frac{1}{n!}  
 	\sum_{k=1}^n \frac{{\rm d}^n}{{\rm d}t^n} \big(1-a_0^{-1}t^m\P(t^{-1})\big)^k_{\vert_{t=0}}.
 	$$
Indeed, by using the Vieta formulas \eqref{e-classic-Viete} we see that
 	\begin{align*}  
 		1-a_0^{-1}t^m \P(t^{-1})&=1-a_0^{-1}\sum_{j=0}^m a_jt^j\\
 		&=\E_1(X_m) t -\E_2(X_m) t^2 +\cdots+ (-1)^{m-1} t^m\E_m(X_m) \\
 		&=F_m(t;X_m),
 	\end{align*}
 	where $F_m(t;X_m)$ is defined as in the proof of Proposition \ref{P-H-E}. Hence, 
 	the claimed identity for 
 	$\H_n(X_m)$ follows immediately from \eqref{finitereduction}.
  \end{Rem}

 \section{Proofs of the results}\label{Ths}
 	
  \begin{proof}[Proof of Theorem \ref{T-Viete-3}]
  	Since
  	$ \dsize
  	\P(z)=a_0\prod_{x\in X_m}(z-x)
  	$, it is plain that 
  	(a) and (b) are equivalent.
  	In order to prove the equivalence of (b) and (c), let us set
  	$$
  	\cal A(t):=t^m\P(t^{-1})=\sum_{j=0}^m a_jt^j,
  	\qquad
  	\cal B(t):=t^n\q_n(t^{-1})
  	=\prod_{y\in J_n}(1-yt).
  	$$
 In view of Proposition \ref{L-pol-E}, we can write 
  	$$
  		\mathcal R(t):=\frac{\cal A(t)}{\cal B(t)}=t^{m-n}\frac{\P(t^{-1})}{\q_n(t^{-1})}
  		=
  		\Big(\sum_{j=0}^m a_jt^j\Big)
  		\Big(\sum_{r=0}^{\infty}\H_r(J_n)t^r\Big)
  		=
  		\sum_{k=0}^{\infty}c_kt^k,
  	$$
  	where  
  	$ \dsize
  	c_k:=\sum_{s=0}^{\min(k,m)}a_s\H_{k-s}(J_n)$ for every 
  	$ k\ge0$. 
  	
  	Now, if (b) holds, then $\P(z)=\q_n(z)\cal S(z)$ for some polynomial
  	$\cal S(z)$ of degree $m-n$, so that
  	$ 
  	\mathcal R(t)=t^{m-n}\cal S(t^{-1})
  	$ 
  	is a polynomial of degree at most $m-n$. Consequently,  $c_k=0$ for every $k>m-n$, yielding (c).

  	Conversely, assume that (c) holds true, i.e., 
  	
  	\centerline{$c_k=0$ for every 
  	$k\in\{m-n+1,\ldots,m\}$.}
  	 
  	Since $\cal B(0)=1$, we can set
  	$ 
  	\cal B(t)=1+b_1t+\cdots+b_nt^n 
  	$
  	and compare the coefficients in $\cal B(t)\mathcal R(t)=\cal A(t)$. 
  	For every integer $k>m$  we have that 
  	$$
  	c_k+b_1c_{k-1}+\cdots+b_nc_{k-n}=0
  	$$
  	because the degree of $\cal A(t)$ is at most $m$. In particular,
  	for $k=m+1$ this is
  	$$
  	c_{m+1}+b_1c_m+\cdots+b_nc_{m-n+1}=0,
  	$$
  	yielding $c_{m+1}=0$ because $c_m=c_{m-1}=\ldots=c_{m-n+1}=0$ by hypothesis. Thus, we can proceed by induction to show that   $c_k=0$ for every $k>m$.
  	Together with the hypothesis (c), this implies that
  	$$
  	\mathcal R(t)=\sum_{k=0}^{m-n}c_kt^k.
  	$$ 
  	Since
  	$ \dsize 
  	\frac{\P(z)}{\q_n(z)}
  	=z^{m-n}\mathcal R(z^{-1}),
  	$ 
  	we can finally  write
  	$$
  	\P(z)
  	=
  	\q_n(z)\sum_{k=0}^{m-n}c_kz^{m-n-k},
  	$$
  	yielding  (b).  
  \end{proof}

\begin{proof}[Proof of Corollary \ref{C-h-sections}] It is plain that if 	$ 
	\cal G(z)$ divides $\cal P_{h,s}(z)$
	for every $s\in\{0,1,\ldots,h-1\}$, then 	$\cal G(z^h)$ divides $\cal P(z)$. 
	
	Conversely,
	assume that  $\cal P(z)=\cal G(z^h)\cal T(z)$  for some complex polynomial $\cal T(z)$. Thus, we can write 
	$$ 
	\cal P(z)=\sum_{s=0}^{h-1}z^s\cal P_{h,s}(z^h)
	=\sum_{s=0}^{h-1}
	z^s\cal G(z^h)\cal T_{h,s}(z^h),
	$$
	where $\cal T_{h,s}(z)$ is defined in an analogous way as $\cal P_{h,s}(z)$ in \eqref{e-def-Ps}.
	Hence, the polynomial 
	$$
	\cal F(z):=\sum_{s=0}^{h-1}\big(\P_{h,s}(z^h)-\cal G(z^h)\cal T_{h,s}(z^h)\big)z^s
	$$
	is identically zero on $\C$, i.e., all its coefficients are equal to zero.
	
	Now, note that
	$\P_{h,s}(z^h)-\cal G(z^h)\cal T_{h,s}(z^h)$ is a polynomial in $z^h$ involving disjoint sets of powers of   $z$ for distinct values of  $s\in\{0,1,\ldots,h-1\}$. Indeed, it suffices to observe that if $0\leq s,\,s'<h$ and $j,\,j'\geq 0$, then
	$
	s+jh=s'+j'h
	$
	if and only if $s=s'$ and $j=j'$. Thus,  being $\cal F(z)$ the zero polynomial, we must conclude that
	$$
	\cal P_{h,s}(z^h)=\cal G(z^h)\cal T_{h,s}(z^h)
	\ \hbox{for every}\   s\in\{0,1,\ldots,h-1\},
	$$
	showing that $\cal G(z)$ divides $\P_s(z)$ for every $s=0,\ldots,h-1$, as required.
\end{proof}

 \begin{proof}[Proof of Corollary \ref{C-divis-cond}]  	Let $J_n$ be the   roots multiset  of $\q$, so that 
 	$\E_j(J_n)=(-1)^jb_j$ by  the Vieta formulas \eqref{e-classic-Viete}. Thus, 
	by  using Proposition \ref{P-H-E} we have that
		\begin{align*}
		\H_\ell(J_n) &= 
		\sum_{s=1}^\ell (-1)^s\hskip-3mm\sum_{\substack{\ell_1+\cdots+\ell_s=\ell\\ 1\le \ell_i\le n}}
		b_{\ell_1}\cdots b_{\ell_s}
		\\ 
		&=\sum_{\substack{\vec\alpha\in\mathbb N_0^n\\\w(\vec\alpha)=\ell}}
		(-1)^{\Sigma(\vec\alpha)}
		\binom{\Sigma(\vec\alpha)}{\alpha_1,\ldots,\alpha_n}
		\prod_{j=1}^{n}b_j^{\alpha_j}.
	\end{align*}
	Hence, the claimed formulas follow
	 from Theorem \ref{T-Viete-3} for $\q(z)$ divides $\P(z)$ if and only if 
	$$ 
	\sum_{\ell=0}^{k}a_{k-\ell}\H_\ell(J_n)=0
	$$ 
	for every $k\in\{m-n+1,\ldots,m\}$.
\end{proof} 
 
\begin{proof}[Proof of Proposition \ref{P-product2poly}]
	Set
	$$
	\cal A(t):=t^m\P(t^{-1})=\sum_{j=0}^m a_jt^j\quad \hbox{and}\quad 
	\cal S_r(t):=\sum_{s=0}^r\H_s(X_m)t^s.
	$$
	By using Proposition \ref{L-pol-E} we write 
	$$
	\cal A(t)H(t;X_m)=A(t) \Big(\cal S_r(t)+\sum_{s=r+1}^\infty
	\H_s(X_m)t^s\Big)=1,
	$$
	so that 
	$\dsize
	\cal A(t)\cal S_r(t)
	=1-\cal A(t)\sum_{s=r+1}^\infty
	\H_s(X_m)t^s$. Now, after noticing  that $\cal A(t)\cal S_r(t)$  has degree $r+m$, we infer
	\begin{align*} \cal A(t)\sum_{s=r+1}^\infty
		\H_s(X_m)t^s&=
		\sum_{k=r+1}^{\infty}t^k
		\sum_{j=0}^m a_j\sum_{s=r+1\atop s+j=k}^\infty
		\H_s(X_m)\\
		&=\sum_{k=r+1}^{r+m}t^k
		\sum_{j=0}^m a_j\sum_{s=r+1\atop s+j=k}^\infty
		\H_s(X_m)\\
		&=\sum_{k=1}^{m}t^{k+r}
		\sum_{j=0}^m a_j\sum_{s=r+1\atop s+j=k+r}^\infty
		\H_s(X_m).
	\end{align*}
	 
	Since the conditions  $s\ge r+1$ and $s+j=k+r$ imply that $j\le k-1$, one has 
	$$
	\cal A(t)\cal S_r(t)
	=1-\cal A(t)\sum_{s=r+1}^\infty
	\H_s(X_m)t^s	=1+\sum_{k=1}^{m}c_{k,r}t^{k+r}
	$$
	with 
	$$
	c_{k,r}:=-\sum_{j=0}^{k-1}a_j\H_{k+r-j}(X_m),
	\quad k\in\{1,\ldots,m\}
	$$
	as in \eqref {e-ckr}.
	Substituting $t=z^{-1}$ and multiplying
	by $z^{m+r}$, the previous identity   gives
	$\dsize 
	\P(z)\q_r(z)
	=z^{m+r}+\sum_{k=1}^m
	c_{k,r}z^{m-k},
	$
	as claimed.
	
	Now, let us take $s\in\C$ fixed and write 
	$$
	z^{m+r}-s
	=\P(z)\q_r(z)-\bigl(\cal R_r(z)+s\bigr),
	$$
	where plainly $\deg(\cal R_r+s)<m$. Thus, we see that $\P(z)$
	divides $z^{m+r}-s$ if and only if
	$\cal R_r(z)=-s$. In view of \eqref {e-ckr}, this  is equivalent to saying that  $c_{m,r}=-s$ and
	$c_{k,r}=0$ when $1\le k<m$.
\end{proof}

\begin{proof}[Proof of Proposition \ref{P-Kronecker-divisibility}]
	(a)$\Rightarrow$(b): suppose that $\cal K(z)$ is a Kronecker polynomial.
	Since $\cal K(0)\ne 0$, by
	Kronecker's theorem we deduce that all the roots of $\cal K(z)$ are roots of unity. Assume  first that  all these roots are simple and let $N$ be a common
	multiple of their orders. If $t>1$ is such that
	$t\equiv 1\pmod N$, 
	then $\alpha^t=\alpha$ 
	for every root $\alpha$ of $\cal K(z)$. Thus, we
	deduce that $\alpha$ is also a root of $\cal K(z^t)$. Consequently,
	$\cal K(z)$ divides $\cal K(z^t)$ because its roots are simple.
	
	Now, assume that  $\alpha_1,\alpha_2,\ldots,\alpha_s$ are the roots of $\cal K(z)$ which are not simple, let $d_1,d_2,\ldots,d_s\ge 2$ be their respective multiplicity.  Since all the roots of the polynomial 
	$$ 
	\cal Q(z) :=\frac{\cal K(z)}{(z-\alpha_1)^{d_1-1}\cdots (z-\alpha_s)^{d_s-1}}
	$$
	are roots of unity and simple,
	by the first part of the proof it follows that  $\cal Q(z)$ divides $\cal Q(z^t)$, where $t$ is defined as before. 
	Further,  
	$(z-\alpha_j)^{d_j-1}$ divides $(z^t-\alpha_j)^{d_j-1}$ for $\alpha_j$ is a root of $z^t-\alpha_j$ by definition of $t$.
	Thus, we conclude that
	$\cal K(z)= \cal Q(z) (z-\alpha_1)^{d_1-1}\cdots (z-\alpha_s)^{d_s-1}$ divides $\cal K(z^t)=\cal Q(z^t) (z^t-\alpha_1)^{d_1-1}\cdots (z^t-\alpha_s)^{d_s-1}
	$, as claimed.
	
	(b)$\Rightarrow$(a): 	assume  that there exists an integer $t>1$ such that
	$\cal K(z)$ divides $\cal K(z^t)$. This implies that if
	$\alpha$ is a root of $\cal K(z)$, then
	so is $\alpha^t$. Further,  it must be $\alpha\ne0$ by the hypothesis 	$\cal K(0)\ne0$. Thus, we can iterate this argument and conclude that 
	$ 
	\alpha,\alpha^t,\alpha^{t^2},\ldots
	$ 
	are roots of $\cal K(z)$. Since there are only finitely
	many distinct roots, it must be
	$ 
	\alpha^{t^r}=\alpha^{t^s}
	$ 
	for some  distinct integers $r,s\in\N_0$, implying that 
	$\alpha$ is a root of unity. Hence, $\cal K(z)$ is
	a Kronecker polynomial.

	(a)$\Rightarrow$(c):  if $ 
	\cal K(z)$  is a Kronecker polynomial, then it admits a representation as	\eqref{e-Kronecker-factorization}, with $r=0$ for $\cal K(0)\ne 0$, namely 
	$$
	\cal K(z)
	=
	\prod_{j\in K}\Phi_j(z)^{b_j},
	$$
	where $K\subset\mathbb N$ is finite and
	$b_j\in\mathbb N$. Note that if $\alpha$ is a root of 
	$\cal K(z)$, then $\alpha^j=1$ for some $j\in K$.
	Therefore, denoting by $N$  the  least common multiple of the integers $jb_j$, with $j\in K$, 
	we deduce that all the roots of $\cal K(z)$ are roots of $z^N-1$.

	Since all the roots of any cyclotomic polynomial are simple, we deduce that 
	the multiplicity of any root of $\cal K(z)$ is not greater than $N$, so that
	$\cal K(z)$ must divide $(z^N-1)^s  $ for some $s\in\N$.
	
	(c)$\Rightarrow$(a):	assume  that there exist $s\in\N$ and an integer $N>1$ such that
	$ 
	\cal K(z)$ divides $(z^N-1)^s$. Thus, if
	$\alpha$ is a root of $\cal K(z)$, then it must be $\alpha^N  =1$, i.e., 
	$\alpha$ is a root of unity. Hence, $ 
	\cal K(z)$  is
	a Kronecker polynomial.
\end{proof}

\begin{proof}[Proof of Theorem \ref{T-kronekerequiv-alt}]
	(a)$\Leftrightarrow$(b):
	by Proposition~\ref{P-Kronecker-divisibility},
	$\mathcal K(z)$ is a Kronecker polynomial if and only if $\mathcal K(z)$ divides $\mathcal K(z^t)$
	for some integer $t>1$. 
	
	Thus, we can apply  Theorem\ref{T-Viete-3} by taking 	$\mathcal Q_n(z)=\mathcal K(z)$ and
	$$ 
	\mathcal P(z):=\mathcal K(z^t)
	=\sum_{r=0}^{tn}b_r z^{tn-r},
	$$
	with
	$$
	b_r:=
	\begin{cases}
		a_j,& \mbox{if $r=tj$ for some $j\in\{0,1\ldots, n\}$},\\
		0,& \mbox{otherwise}.
	\end{cases}
	$$
	We deduce that the divisibility of $\mathcal K(z^t)$ by $\cal K(z)$ is equivalent to  
	$$
	0=\sum_{r=0}^{k}b_r\H_{k-r}(X_n)
	=\sum_{j=0}^{[k/t]}
	a_j\H_{k-tj}(X_n)
	$$
	for every 	$k\in\{(t-1)n+1,\ldots, tn\}$,
	yielding \eqref{e-sum-K}.

	(b)$\Leftrightarrow$(c):  first, note that
	$\H_0(X_n)=F_0=1$. Then, let us apply
	Proposition \ref{P-H-E}   and \eqref{e-classic-Viete}  to see that 	
	\begin{align*} \H_k(X_n) &=	\sum_{s=1}^k (-1)^{k-s}\hskip-4mm\sum_{\substack{k_1+\cdots+k_s=k \\k_i\in\N}}
		\E_{k_1}(X_n )\cdots\E_{k_s}(X_n  ) 
		\\ &	=
		\sum_{s=1}^k (-1)^s \sum_{\substack{k_1+\cdots+k_s=k \\k_i\in\N}}
		a_{k_1} \cdots a_{k_s} =F_k
	\end{align*}
	for every $k\in\N$.
	Hence,
	\eqref{e-sum-K} and \eqref{e-Fk=0} are equivalent.
\end{proof}

 \begin{proof}[Proof of Proposition \ref{generalpropertyBorwein}]
	For
	any given $s\in\{1,\ldots, n\}$ we define 
	$$
	\P_s(z):=\prod_{\ell:\, d_\ell\ge s}\Phi_{j_\ell}(z).
	$$
	Notice that $\P(\alpha)=0$ if and only if there exists a unique $\ell\in\{1,\ldots, k\}$ such that $\Phi_{j_\ell}(\alpha)=0$. Further, if this is the case, then $d_\ell$ is the multiplicity of $\alpha$ as a root of $\P(z)$ and $\P_s(\alpha)=0$  for every $s\in\{1,\ldots, d_\ell\}$. Thus, $d_\ell$ is also the multiplicity of $\alpha$ as a root of $\P_1(z)\cdots \P_n(z)$, so that it must be 
	$$
	\P(z)=\P_1(z)\cdots \P_n(z).
	$$
	Since for every given $s\in\{1,\ldots, n\}$ one has that $P_s(z)$ divides
	$z^{L_s}-1$, where $L_s:=\operatorname{lcm}\{j_\ell:\, d_\ell\ge s\}$, we choose $N_1=L_1$ and $N_s\in\N$  as a  multiple of $L_s$ when $2\le s\le n$ and such  that \eqref{e-ineq1} is satisfied. 
	Hence, we conclude that $\P(z)$ divides $\cal B(z)$, as claimed.
	
	In order to see that $\cal B(z)$ is a Borwein polynomial, first let us write
	$$
	\cal B(z)=
	\sum_{S\subseteq\{1,\ldots,n\}}
	(-1)^{n-|S|}z^{k_S},\ \hbox{where}\ k_S:=
	\begin{cases}
		\dsize\sum_{s\in S}N_s,& \mbox{if  $S\not=\emptyset$},\\
		0,& \mbox{otherwise.}
	\end{cases}
	$$
	Then, observe that
	the inequalities \eqref{e-ineq1}  ensure that
	distinct subsets $S$ of $\{1,\ldots,n\}$ give distinct exponents $k_S$. 
	Hence, all coefficients of $\cal B(z)$ belong to $\{-1,0,1\}$. Moreover,
	$\cal B(0)=(-1)^n\ne0$, as required.
\end{proof}

\end{document}